\documentclass[12pt]{article}
\usepackage{amsmath}
\usepackage{amssymb}
\usepackage{amsmath,amsthm,amscd,amsfonts,amssymb,mathrsfs}
\usepackage[usenames]{color}

\usepackage{cite}
\usepackage{titlesec}

\allowdisplaybreaks

\numberwithin{equation}{section}
\titleformat*{\section}{\large\bfseries}

\newtheorem{thm}{Theorem}[section]
\newtheorem{lem}[thm]{Lemma}
\newtheorem{cor}[thm]{Corollary}
\newtheorem{prop}[thm]{Proposition}

\newtheorem{rem}[thm]{Remark}

\begin{document}
\begin{center}
Smoothness of integrated density of states and level statistics of the Anderson model when single site distribution is convolution with the Cauchy distribution\\~\\
Dhriti Ranjan Dolai\\
Indian Institute of Technology Dharwad \\
Dharwad  - 580011, India.\\
Email: dhriti@iitdh.ac.in
\end{center}
{\bf Abstract:} 
 In this work we consider the Anderson model on $\ell^2(\mathbb{Z}^d)$ when the single site distribution (SSD) is given by $\mu_1 * \mu_2$, where $\mu_1$ is the Cauchy distribution and $\mu_2$ is any probability measure. For this model we prove that the integrated density of states (IDS) is infinitely differentiable irrespective of the disorder strength.  
Also, we investigate the local eigenvalue statistics of this model in $d\ge 2$, without any  assumption on the localization property.\\~\\
 {\bf MSC (2020):} 81Q10, 47B80, 35J10, 35P20.\\
{\bf Keywords:} Anderson Model, random Schr\"{o}dinger operators, integrated density of states, eigenvalue statistics.
\section{Introduction}
The Anderson Model is a random Hamiltonian $H^\omega$ on $\ell^2(\mathbb{Z}^d)$ defined by
\begin{align}
\label{model}
H^\omega &=\Delta+ V^\omega,~~~\omega\in\Omega,\\
(\Delta u)(n) &=\sum_{|k-n|=1}u(k),~ u=\{u(n)\}_{n\in\mathbb{Z}^d}\in\ell^2(\mathbb{Z}^d),\nonumber\\
(V^\omega u)(n) &=\omega_nu(n) \nonumber,
\end{align}
where $\{\omega_n\}_{n\in\mathbb{Z}^d}$ are i.i.d real random variables whose common distribution is $\mu=\mu_1*\mu_2$.  Here $\mu_1$ be the Cauchy distribution whose density is given by $d\mu_1(x)=\frac{\lambda}{\pi(x^2+\lambda^2)}dx,~\lambda>0$ and $\mu_2$ can be any probability measure on $\mathbb{R}$. Consider the probability space $\big(\mathbb{R}^{\mathbb{Z}^d}, \mathcal{B}_{\mathbb{R}^{\mathbb{Z}^d}}, \mathbb{P} \big)$, where $\mathbb{P}=\underset{n\in\mathbb{Z}^d}{\otimes}\mu$ is constructed via the Kolmogorov theorem. We refer to this probability space as $\big(\Omega, \mathcal{B}_\Omega, \mathbb{P}\big)$ and denote $\omega=(\omega_n)_{n\in\mathbb{Z}^d}\in \Omega$.
The operator $\Delta$ is known as the discrete Laplacian and the potential $V^\omega$ is the multiplication operator on $\ell^2(\mathbb{Z}^d)$ by the sequence $\{\omega_n\}_{n\in\mathbb{Z}^d}$. We
note that the operators $\{H^\omega\}_{\omega\in\Omega}$ are self-adjoint and have common core domain consisting of vectors with finite support.
It is well known (see \cite[Theorem 3.9]{Kir}) that  the spectrum of the operator $H^\omega$ is full real line i.e $\sigma(H^\omega)=\mathbb{R}$ a.e $\omega$, as the support of $\mu$, the single site distribution (SSD) is $\mathbb{R}$.  \\~\\
Denote $\chi_{_L}$ to be  the orthogonal projection onto $\ell^2(\Lambda_L)$. Here $\Lambda_L\subset \mathbb{Z}^d$ denote the cube center at origin of side length $2L+1$, namely 
$$\Lambda_L=\big\{ n=(n_1,n_2,\cdots, n_d)  \in\mathbb{Z}^d : |n_i|\leq L,~1\leq i\leq d\big\}.$$
We define the matrix $H^\omega_L$ of $(2L+1)^d$ as 
\begin{equation}
\label{restr}
H^\omega_L=\Delta_L+V^\omega_L,~\Delta_L=\chi_{_L}\Delta\chi_{_L} ,~V^\omega_L=\chi_{_L}V^\omega\chi_{_L}.
\end{equation}
Since the spectrum of $H^\omega_L$  consists of discrete eigenvalues we can define the eigenvalue counting function as
\begin{equation}
\label{count}
\mathcal{N}_L^\omega(E)=\#\bigg\{j: E_j\leq E,~E_j\in\sigma(H^\omega_L)\bigg\},~~E\in\mathbb{R}.
\end{equation}
Let $\mathcal{N}(\cdot)$ is the integrated density of states (IDS) and $\nu(\cdot)$ denote the density of states measure (DOSm) of $H^\omega$, then the definition of the IDS will give
\begin{equation}
\label{ids}
\lim_{L\to\infty}\frac{\mathcal{N}_L^\omega(E)}{(2L+1)^d}=\mathcal{N}(E)~a.e~\omega~~\text{and}~~
\mathcal{N}(E)=\nu(-\infty, E],~~E\in\mathbb{R}.
\end{equation}
Since the single site distribution (SSD) is given by $\mu=\mu_1*\mu_2$ and $\mu_1$ is the Cauchy distribution then it is easy to verify that $\mu$ is absolutely continuous 
w.r.t the Lebesgue measure on $\mathbb{R}$ with bounded density. Now the Wegner estimate (given in (\ref{wm})) will ensure that $\nu$, the density of states measure (DOSm) is also absolutely continuous w.r.t the Lebesgue measure and it has a bounded density say $\rho$, i.e $d\nu(x)=\rho(x)dx$ or in  other words $\rho(x)=N'(x)$ a.e $x$ (w.r.t Lebesgue measure). The density $\rho(\cdot)$ of the measure $\nu(\cdot)$ is known as the density of states function (DOSf) of $H^\omega$. \\~\\
We also define another random Schr\"{o}dinger operator $h^{\omega_2}$  on $\ell^2(\mathbb{Z}^d)$ as
\begin{equation}
\label{h}
(h^{\omega_2}u)(n)=\sum_{|k-n|=1}u(k)+\omega_{2,n} ~u(n),~~\{u(n)\}_{n\in\mathbb{Z}^d}\in\ell^2(\mathbb{Z}^d),
\end{equation}
where $\{\omega_{2,n}\}_{n\in\mathbb{Z}^d}$ are i.i.d real random variables distributed by $\mu_2$. Here also, we consider the product probability space  
$\bigg((supp~\mu_2)^{\mathbb{Z}^d}, \mathcal{B}, \underset{n\in\mathbb{Z}^d}{\otimes}\mu_2\bigg)$
and denote $\omega_2=(\omega_{2,n})_{n\in\mathbb{Z}^d}\in (supp~\mu_2)^{\mathbb{Z}^d}$ .\\~\\
With all these notations in place, we state our main results.
\begin{thm}
\label{thm1}
The Fourier transformation of $\nu$, the density of states measure (DOSm) of $H^\omega$ is given by
\begin{equation}
\label{ft}
\hat{\nu}(t)=e^{-\lambda |t|} ~\mathbb{E}\bigg( \big\langle e_0, e^{-ith^{\omega_2}} e_0 \big\rangle\bigg),
\end{equation}
here $\{e_n\}_{n\in\mathbb{Z}^d}$ denote the standard basis of $\ell^2({\mathbb{Z}^d})$ and $h^{\omega_2}$ as in  (\ref{h}).
\end{thm}
\noindent Since $\nu(\cdot)$, the DOSm of $H^\omega$ is absolutely continuous, i.e $d\nu(x)=\rho(x) dx$, then the properties of the Fourier transform of convolution will give the differentiability of $\rho$.
\begin{cor}
\label{difernblty}
The density of states function (DOSf )  $\rho(x)$ of $H^\omega$ can be written as
\begin{equation}
\label{conv}
\rho(x)=(g*\nu_2)(x),
\end{equation}
here $\nu_2$ is the DOSm of $h^{\omega_2}$, defined by (\ref{h}) and $g(x)=\frac{1}{\pi}\frac{\lambda}{x^2+\lambda^2},~\lambda>0$ is the density of the Cauchy distribution.  Now the infinite differentiability of $\rho$ is immediate from $(\ref{conv})$.
\end{cor}
\noindent Let's define the two functions spaces $\mathcal{F}_{ac}$ and $\mathcal{F}_{conv}$ as
\begin{equation*}
\mathcal{F}_{ac}:=\bigg\{h\in L^1(\mathbb{R}):~h~\text{is the DOSf of} ~H^\omega ~\text{with absolutely continuous SSD} \bigg\}.
\end{equation*}
\begin{equation*}
\begin{split}
& \mathcal{F}_{conv}:=\bigg\{ \rho \in L^1(\mathbb{R}):~\rho~\text{ is the DOSf of}~H^\omega~\text{with SSD} ~\mu=\mu_1*\mu_2, \\
&\qquad \qquad  \qquad~\text{here}~\mu_1~\text{is the Cauchy distribution with parameter}~ \lambda>0 \\
&\qquad \qquad \qquad ~~ \text{and}~\mu_2~\text{is any probability measure}\bigg\}.
 \end{split}
\end{equation*}
From the above definitions it is clear that $\mathcal{F}_{conv}\subset \mathcal{F}_{ac}\subset L^1(\mathbb{R})$. Also, as a corollary to the above theorem we observe that any DOSf  of $H^\omega$ with absolutely continuous SSD can be approximate by a sequence of infinite differentiable DOSfs, in $L^1(\mathbb{R})$.
\begin{cor}
\label{appro}
The space  $\mathcal{F}_{conv}$ is dense subset of $\mathcal{F}_{ac}$ in $L^1(\mathbb{R})$ norm and also we have $\mathcal{F}_{conv}\subset C^\infty (\mathbb{R})$.
\end{cor}
\noindent We can also apply our method to show the regularity properties of expected empirical spectral distribution (EESD) of certain kind of random matrices.\\~\\
 Let $A_N=\big(a_{i j} \big)_{1\leq i, j\leq N}$ is a symmetric matrix of order $N$ whose elements are real random variables. Let $D_N=\text{diag}\{d_i\}_{1\leq  i\leq N}$ be a diagonal matrix of order $N$ and its diagonal entries  $\{d_i\}_{1\leq i\leq  N}$ are i.i.d random variables distributed by the Cauchy distribution with parameter $\lambda>0$.\\
Let $F_N(\cdot)$ denote the empirical spectral distribution (ESD) of $A_N$ and $\tilde{F}_N(\cdot)$ denote the same for the matrix $A_N+D_N$. Therefore, we set
\begin{equation}
\label{esd}
\begin{split}
F_N(x):&=\frac{\#\{j: E_j\leq x,~E_j\in\sigma(A_N)\}}{N},~x\in\mathbb{R}\\
\tilde{F}_N(x):&=\frac{\#\{j: E_j\leq x,~E_j\in\sigma(A_N+D_N)\}}{N},~x\in\mathbb{R}.
\end{split}
\end{equation}
Let $L_N(\cdot)$ and $\tilde{L}_N(\cdot)$ denote the measures corresponding to the expected  empirical spectral distribution of $A_N$ and $A_N+D_N$, respectively and they are given by
\begin{equation}
\label{eesd}
\begin{split}
L_N(-\infty, x]:=\mathbb{E}\big(F_N(x)\big)~~\text{and}~~\tilde{L}_N(-\infty, x]:=\mathbb{E}\big(\tilde{F}_N(x)\big).
\end{split}
\end{equation}
\begin{cor}
\label{eesd-conv}
Let $A_N$ and $D_N$ are two independent random matrices as defined above. Then the expected empirical spectral distribution (EESD) of $A_N+D_N$ is infinite differentiable and it is given by
\begin{equation}
\label{conv-eesd}
\frac{d^k\tilde{\rho}_N}{dx^k}(x)=\bigg(\frac{d^k g}{dx^k}*L_N \bigg)(x)~\forall~k\in\mathbb{N}\cup\{0\},~\tilde{\rho}_N(x)=\frac{d}{dx}\mathbb{E}(\tilde{F}_N(x)),
\end{equation}
here the measure $L_N(\cdot)$ is given in (\ref{eesd}) and $g(x)=\frac{1}{\pi}\frac{\lambda}{x^2+\lambda^2},~\lambda>0$.
\end{cor}
\begin{rem}
The above corollary also guarantee that if the limit of $\mathbb{E}\big(F_N(\cdot)\big)$ exists, as $N$ gets larger, then the same is true for $\mathbb{E}\big(\tilde{F}_N(\cdot)\big)$ and in that case the limit of $\mathbb{E}\big(\tilde{F}_N(\cdot)\big)$ will be infinite differentiable.\\
In other words any real symmetric random matrix of the form $A_N+D_N$ always have infinite smooth expected empirical spectral distribution (EESD), as long as $A_N,~D_N$ are independent and $D_N$ is a diagonal matrix whose entries are i.i.d Cauchy. We also note that there is no restriction on the distributions of the elements of $A_N$.
\end{rem}
\noindent Now we consider the rescale matrix
\begin{equation}
\label{rescale}
H^\omega_{\gamma,E,L}=(2L+1)^\gamma \big(H^\omega_L-E\big),~\gamma>0,~E\in\mathbb{R}.
\end{equation}
 We intend to study the limit of the sequence of random measures $\{\mu^\omega_{E, L}(\cdot)\}_L$ as $L$ gets large, where the random measure is defined as
\begin{align}
\label{meas1}
\mu^\omega_{E,L}\big(\cdot\big) &=\frac{1}{(2L+1)^\beta}Tr\bigg(E_{H^\omega_{\gamma, E,L}}\big(\cdot\big) \bigg),~~\gamma+\beta=d\nonumber\\
                                 &=\frac{1}{(2L+1)^\beta}\sum_{E_j\in\sigma(H^\omega_L)}\delta_{_{(2L+1)^\gamma(E_j-E)}}\big(\cdot\big),~~E\in\mathbb{R}.
\end{align}
In the above $E_A(\cdot)$ denote the spectral measure of a self adjoint  operator $A$ and $\delta_b(\cdot)$ be the Dirac measure at the point $b$.\\~\\
\noindent Now we will describe the convergence of the sequence of random measures $\{\mu^\omega_{E, L}(\cdot)\}_L$ as defined in (\ref{meas1}), associated with eigenvalues of the matrix $H^\omega_L$. We observe that the total mass $\mu^\omega_{E,L}(\mathbb{R})$ increases to infinity as $L$ gets larger, but the Wegner estimate (\ref{wm}) imply $\displaystyle\sup_{L}\mathbb{E}\big(\mu^\omega_{E,L}(K)\big)<\infty$ for each compact set $K\subset{\mathbb{R}}$. Therefore, we will talk about the vague convergence of the sequence of random measures $\{\mu^\omega_{L,E}(\cdot)\}_L$ and try to find its limit.
\begin{thm}
\label{eig-sta}
Consider the random measure $\mu^\omega_{E, L}(\cdot)$ defined in (\ref{meas1}) with the parameters $0<\gamma< \frac{d-1}{2d}$ and $\beta=d-\gamma$, then for $d\ge 2$ we have
\begin{equation}
\label{weak-conv}
\mu^\omega_{E, L}(\cdot)\xrightarrow[L\to\infty]{vaguely}\rho(E)\mathcal{L}(\cdot)~~a.e~\omega,~~\text{whenever}~\rho(E)>0,
\end{equation}
where $\mathcal{L}(\cdot)$ denote the Lebesgue measure on $\mathbb{R}$. 
\end{thm}
\begin{rem}
We say the sequence of measure $\{\mu_n\}_n$ converges vaguely to a measure $\mu$ on $\mathbb{R}$ iff
$\int_{\mathbb{R}}f(x)d\mu_n(x)\xrightarrow{n\to\infty} \int_{\mathbb{R}}f(x)d\mu(x)$, for all $f\in C_c(\mathbb{R})$, the set  of all continuous function on $\mathbb{R}$ with compact support.
\end{rem}
\noindent More details about the vague convergence of measures on complete, separable metric space can be found in Kallenberg \cite[Section 4.1]{KO}.\\~\\
The above theorem also assure that if we count the number of eigenvalues of $H^\omega_L$ in an interval around a point $E$ and normalized it by the size of the matrix $H^\omega_L$,  now if we shrink the interval as $L$ gets larger then in the limit we will capture $\rho(E)$, the value of the density of states at the point $E$.
\begin{cor}
\label{cap-dos}
Let $E\in\mathbb{R}$ such that $\rho(E)>0$, then for $0<\gamma<\frac{d-1}{2d},~d\ge 2$ and $-\infty<a<b<\infty$ we have
\begin{equation}
\label{-dos-cap}
\lim_{L\to\infty}\frac{\#\big\{j:E_j\in I_{a,b,L},~E_j\in\sigma(H^\omega_L) \big\}}{(b-a)(2L+1)^{d-\gamma}}=\rho(E)~~a.e~\omega.
\end{equation}
Here $I_{a,b,L}$ denote the interval $\big[E+\frac{a}{(2L+1)^\gamma}, E+\frac{b}{(2L+1)^\gamma}\big]$.
\end{cor}
\begin{rem}
We observe that the proofs of all the above results will also work when the single site distribution (SSD) is only the Cauchy distribution, i.e $\nu=\mu_1$. In that case the operator, $h^{\omega_2}$, in (\ref{h}) is nothing but the free Laplacian, i.e $\omega_{2}\equiv 0$ and $h^{\omega_2}=h^0:=\Delta$. 
\end{rem}
\begin{rem}
We note that all the methods which will be used in the proofs will also work for the Anderson model on Bethe lattice with the same single site distribution (SSD) $\mu=\mu_1*\mu_2$, here $\mu_1$ is the Cauchy distribution and $\mu_2$ can be any probability measure on $\mathbb{R}$.
\end{rem}
\noindent The above theorem is motivated by the local statistics of the spectrum considered for the Anderson model in the region of exponential localization. The study of eigenvalue statistics was done by Molchanov \cite{mol} in one-dimension and by Minami \cite{Minami} in higher dimensional Anderson model. They showed that if we consider the random measure $\mu^\omega_{E,L}(\cdot)$ with $\gamma=d$, $\beta=0$ and $E$ lies in the localization region then  $\mu^\omega_{E,L}$ converges weakly to the Poisson point process as $L\to\infty$, provided $\rho(E)>0$. On the other hand, for $\gamma=0$ and $\beta=d$, the convergence of $\mu^\omega_{L,E}(\cdot)$ is the definition of the integrated density of states (IDS).
Klopp \cite{Klopp}  proved the Poisson limit theorem for unfolded eigenvalues of random operators in localized regime, see also Germinet-Klopp \cite{GK}.
Subsequently the Poisson statistics was shown for the trees by Aizenman-Warzel in \cite{AW} and for the regular graphs by Geisinger \cite{Ge}. 
In  \cite{KKH} Kirsch-Krishna-Hislop considered the Anderson model on $L^2({\mathbb{R}^d})$ in $d\leq 3$ and obtained the Poisson statistics, where potentials is formed from the delta interactions at lattice points of $\mathbb{Z}^d$ and random coupling constants. For continuum alloy-type model on $L^2(\mathbb{R}^d)$, we refer to Dietlein-Elgart \cite{DE} .\\
All the results described above need assumption on the localization properties to show the eigenvalue statistics is Poisson.\\
For the decaying model on $\ell^2(\mathbb{Z}^d)$, Dolai-Krishna \cite{DK} considered the random measure $\mu^\omega_{E,L}(\cdot)$ with $\alpha=1$, $\beta=d-1$, here $E$ lies in the absolutely continuous spectrum $[-2d,2d]$ and for $d\ge 3$ they proved that $\mu^\omega_{E,L}(\cdot)$ and $\mu^0_{E,L}(\cdot)$ has  same limit point, a.e $\omega$. It was also shown that the sequence of measures $\{\mu^0_{E,L}(\cdot)\}_L$ admits non-trivial limit points, here $\mu^0_{E,L}(\cdot)$ denote the non-random measure associated with the free Laplacian $\Delta$.\\
For one dimensional decaying model some statistics were known inside the continuous spectrum. 
Krichevski-Valk\'{o}-Vir\'{a}g \cite{KVV} showed the Sine-$\beta$ process when the exponent for the decaying coefficient, $\alpha=\frac{1}{2}$ and Dolai-Mallick \cite{DM} proved the clock process for $\alpha>\frac{1}{2}$. For one-dimensional continuum decaying model, Nakano \cite{Na} showed that the statistics for $\alpha>\frac{1}{2}$ is clock and for $\alpha=\frac{1}{2}$, it is circular $\beta$-ensemble, see also Kotani \cite{Ko}. Avila-Last-Simon \cite{ALS} showed quasi-clock behaviour for ergodic Jacobi operator in region of absolutely continuous spectrum and Breuer-Weissman\cite{BW} showed the strong level repulsion (uniform clock behaviour) for one dimensional continuum model with purely singular continuous spectrum.\\ ~\\
However, in the absence of localization no statistics was investigated for the higher dimensional  Schr\"{o}dinger operator with stationary potential. \\~\\
As far as smoothness of IDS is concern, various results are known in the localized regime. Most recently, Dolai-Krishna-Mallick \cite{DKM} proved that the IDS is as smooth as the single site distribution (SSD) of the Anderson Model on $L^2(\mathbb{R}^d)$, in the presence of localization.\\
For the Anderson model on $\ell^2(\mathbb{Z}^d)$, Constantinescu-Fr\"{o}hlich-Spencer \cite{CFS} showed that the IDS is analytic whenever SSD is analytic, see also Carmona \cite[Corollary VI.3.2]{CL}. In \cite{KKS}, Kaminaga-Krishna-Nakamura showed local analyticity of IDS whenever the SSD has an analytic component. If the Fourier transform of SSD is $C^\infty$ function and its derivative decay first enough, then smoothness of IDS was proved by Bovier-Campanino-Klein-Perez \cite{BCKP}, see also Bellissard-Hislop \cite{BH}. \\
There are many results on the smoothness of IDS for one-dimensional case, we refer to Companino-Klein \cite{CK}, March-Sznitman \cite{MS} and Simon-Taylor \cite{ST} for more details.  Smoothness results were shown by Speis \cite{S}, Klein-Speis \cite{KS, KS1}, Klein-Lacroix-Speis \cite{KLS} and Glaffig \cite{G} for the Anderson model on one dimensional strip. \\
In all the results mentioned above, the higher order differentiability of $\mathcal{N}(E)$ is shown when $E$ varies in the pure point spectrum.\\
On the other hand, without any localization or high disorder assumptions very few results are known so far about the smoothness of IDS. 
In \cite{AK}, Acosta-Klein proved the smoothness of IDS on Bethe lattice in the region of absolutely continuous spectrum when the SSD is Cauchy or close to Cauchy, in a function space. On $\ell^2(\mathbb{Z}^3)$, Lloyd \cite{L} computed the exact expression of IDS with Cauchy distribution as the SSD. The same expression of IDS on $\ell^2(\mathbb{Z}^d)$ is given in  the book of
 Carmona-Lacroix, see \cite[problem VI.5.5]{CL}. Kirsch-Krishna \cite{KK} showed that the exact expressions of the IDS is also valid for some continuous models and as well as in Bethe lattice whenever SSD is given by the Cauchy distribution.
 \section{Proof of the results}
In this section we give the proofs of the Theorem \ref{thm1} and \ref{eig-sta}. Before going to the proofs we will show that local density of states function ($\ell$DOSf) converges uniformly to the density of states function (DOSf). This uniform convergence of local density will play a crucial role to show the vague convergence of the sequence of measures $\big\{\mu^\omega_{E,L}(\cdot)\big\}_L$, as  defined in (\ref{meas1}). 
Let's begins by defining the local density of states measure ($\ell$DOSm) as
\begin{equation}
\label{ldosm}
\nu_L\big(\cdot\big)=\frac{1}{(2L+1)^d}\mathbb{E}\bigg(Tr \bigg(E_{H^\omega_L}\big(\cdot\big)\bigg)\bigg).
\end{equation}
The Wegner estimate, \cite [Theorem 2.3]{JFA} ensure that the measure, $\nu_L(\cdot)$ on $\mathbb{R}$ is absolutely continuous (w.r.t Lebesgue measure) and it has a bounded density $\rho_L$, in other words $d\nu_L(x)=\rho_L(x) dx$. The density function $\rho_L$ is known as the local density of states function ($\ell$DOSf).  Since $\nu_L$ converges weakly to $\nu$, so we have the pointwise convergence of the Fourier transform of their respective densities i.e $\hat{\rho}_L(t)\to\hat{\rho}(t)$ as $L\to\infty$.  Therefore to show the expression (\ref{ft}) in Theorem \ref{thm1}, first we calculate the $\hat{\rho}_L(t)$ then we will take the limit.\\~\\
{\bf Proof of Theorem \ref{thm1}. } Since $\{\omega_n\}_{n\in\mathbb{Z}^d}$ are i.i.d real random variables distributed by 
$\nu=\mu_1*\mu_2$, therefore 
$V^\omega_L=\displaystyle\sum_{n\in\Lambda_L}\omega_n~\big|e_n \big\rangle \big\langle e_n\big|$ can be viewed as
\begin{align}
\label{split}
V^\omega_L &=\sum_{n\in\Lambda_L}\big(\omega_{1,n}+\omega_{2,n}\big)~\big|e_n \big\rangle \big\langle e_n\big|\nonumber\\
&=V^{\omega_1}_L+V^{\omega_2}_L.
\end{align}
Here $\{\omega_{1,n}\}_{n\in\Lambda_L}$ and  $\{\omega_{2,n}\}_{n\in\Lambda_L}$ are i.i.d real random variables distributed by $\mu_1$ and $\mu_2$, respectively. Also, we assume that the collection of random variables $\big\{\omega^j_n: j=1,2 \big\}_{n\in\Lambda_L}$ is mutually independent. \\
In view of (\ref{split}) and the definition of $h^{\omega_2}$ as in (\ref{h}), we write
\begin{align}
\label{sum-of-ind}
H^\omega_L &=\Delta_L+V^\omega_L
                    =h^{\omega_2}_L+V^{\omega_1}_L~~~\text{on}~\ell^2(\Lambda_L),
  \end{align}
where $h^{\omega_2}_L$  is the restriction of $h^{\omega_2}$  onto $\ell^2(\Lambda_L)$.  Note that $h^{\omega_2}_L$  and $V^{\omega_1}_L$ are two independent random matrix of order $(2L+1)^d$.\\
Now we use the Trotter product \cite[ Theorem VIII.3]{RS} to write
\begin{equation}\label{tro-pro}
\begin{split}
&\mathbb{E}\bigg(\bigg\langle e_n, ~e^{-itH^\omega_L} ~e_n\bigg\rangle\bigg)\\
&=\lim_{m\to\infty}\mathbb{E}\bigg(\bigg\langle e_n, \bigg(e^{-i\frac{t}{m}h^{\omega_2}_L} e^{-i\frac{t}{m}V^{\omega_1}_L}\bigg)^m ~e_n\bigg\rangle\bigg) \\
&=\lim_{m\to\infty}\mathbb{E}\bigg(\sum_{k_1, k_2\cdots k_{m-1}\in\Lambda_L} e^{-i\frac{t}{m}\big(\omega_{1,n}+\sum_{j=1}^{m-1}\omega_{1,k_j}\big) }~Q^{\omega_2}_{k_1, k_2,\cdots k_{m-1}}\bigg).
\end{split}
\end{equation}
The term $Q^{\omega_2}_{k_1, k_2,\cdots k_{m-1}}$, inside the above summation is given by
\begin{equation}\label{prod}
\begin{split}
&Q^{\omega_2}_{k_1, k_2,\cdots k_{m-1}}=\bigg\langle e_{k_1}, e^{-i\frac{t}{m}h^{\omega_2}_L} e_{n}\bigg\rangle 
\bigg\langle e_{k_2}, e^{-i\frac{t}{m}h^{\omega_2}_L} e_{k_1}\bigg\rangle\\
&\qquad \qquad \qquad \qquad   \times \bigg\langle e_{k_3}, e^{-i\frac{t}{m}h^{\omega_2}_L} e_{k_2}\bigg\rangle
\cdots \cdots \bigg\langle e_{n}, e^{-i\frac{t}{m}h^{\omega_2}_L} e_{k_{m-1}}\bigg\rangle.
\end{split}
\end{equation}
For each index $ (k_1, k_2,\cdots, k_{m-1}), k_j\in\Lambda_L$ in the sum (\ref{tro-pro}), the exponential term can be written as
\begin{equation}
\label{c-dis}
e^{-i\frac{t}{m}\big(\omega_{1,n}+\sum_{j=1}^{m-1}\omega_{1,k_j}\big) }=
e^{-i\frac{t}{m}\big(\sum_{k\in\Lambda_L} \alpha_k\omega_{1,k}\big) },~\sum \alpha_k=m,~~\alpha_k\ge 0.
\end{equation}
Since $\{\omega_{1,n}\}_n$ are i.i.d whose common distribution is Cauchy  with parameter $\lambda>0$, we have
\begin{equation}
\label{c-f}
\mathbb{E}\bigg(e^{-i\frac{t }{m}\alpha_k\omega_{1,k}}\bigg)=e^{-\lambda \frac{\alpha_k}{m}|t|}.
\end{equation}
The exponential and the expression $Q^{\omega_2}_{k_1, k_2,\cdots k_{m-1}}$ in each term of the summation inside the expectation (\ref{tro-pro}) are independent of each other.  We use (\ref{c-f}), (\ref{c-dis}) and (\ref{prod}) in (\ref{tro-pro}) to write
\begin{align}
\label{ldosf-ft}
\mathbb{E}\bigg(\bigg\langle e_n, ~e^{-itH^\omega_L} ~e_n\bigg\rangle\bigg) &=
e^{-\lambda|t|}\lim_{m\to\infty}\mathbb{E}\bigg(\sum_{k_1, k_2\cdots k_{m-1}\in\Lambda_L} Q^{\omega_2}_{k_1, k_2,\cdots k_{m-1}}\bigg)\nonumber\\
&=e^{-\lambda|t|}~\mathbb{E}\bigg(\bigg\langle e_n, ~e^{-ith^{\omega_2}_L} ~e_n\bigg\rangle\bigg) .
\end{align}
We use the above expression to calculate the Fourier transform of the IDS of $H^\omega$ in terms of the Fourier transform of the IDS of $h^{\omega_2}$ as
\begin{align}
\label{cal-ft}
\hat{\nu}_L(t) &=\int e^{-itx}d\nu_L(x)\nonumber\\
                 &=\frac{1}{(2L+1)^d}\sum_{n\in\Lambda_L} \mathbb{E}\bigg(\bigg\langle e_n, ~e^{-itH^\omega_L} ~e_n\bigg\rangle\bigg)\nonumber \\
                 &=e^{-\lambda|t|}\frac{1}{(2L+1)^d}\sum_{n\in\Lambda_L} \mathbb{E}\bigg(\bigg\langle e_n, ~e^{-ith^{\omega_2}_L} ~e_n\bigg\rangle\bigg).
\end{align}
The definition of IDS of $h^{\omega_2}$ will give the (\ref{ft}), once we take the limit, as $L\to\infty$ in both side of the above equation. \qed\\~\\
Now the Corollary {\ref{difernblty}} is a immediate extension of the above theorem and it is described below.\\~\\
{\bf Proof of Corollary {\ref{difernblty}}.} The r.h.s of (\ref{ft}) is the pointwise multiplication of Fourier transform of the Cauchy distribution and Fourier transform of the density of states measure (DOSm) of $h^{\omega_2}$. As $d\nu(x)=\rho(x)dx$,
the inverse Fourier transform of (\ref{ft}) will give (\ref{conv}). Consequently we have 
$$\frac{d^k\rho}{dx^k}(x)=\bigg(\frac{d^k g}{dx^k}*\nu_2\bigg)(x)~\forall~k\in\mathbb{N},$$
here $g$ is the density of the Cauchy distribution with parameter $\lambda>0$. \qed\\~\\
{\bf Proof of Corollary {\ref{appro}}.} The fact $\mathcal{F}_{conv}\subset C^\infty(\mathbb{R})$ will follow from (\ref{conv}). Now assume $h$ is the DOSf of $H^\omega$ whose SSD is an absolutely continuous distribution $\mu_2$. Denote $\rho$ to be the DOSf of $H^\omega$ with SSD $\nu=\mu_1*\mu_2$, here $\mu_1$ is the Cauchy distribution with parameter $\lambda>0$. Using (\ref{conv}) we can write the expression of $\rho(x):=\rho_\lambda(x)$ as
\begin{align}
\label{apro-id}
\rho_\lambda(x)=\frac{\lambda}{\pi}\int_\mathbb{R}\frac{h(x)}{(x-y)^2+\lambda^2}dy.
\end{align}
Now the properties of approximate identity will give the result.
\qed\\~\\
The same techniques used in the proofs of Theorem \ref{thm1} and Corollary {\ref{difernblty}} will be applied here to prove the Corollary \ref{eesd-conv}, so we are omitting the details.\\~\\
{\bf Proof of Corollary {\ref{eesd-conv}}.} Using the spectral theorem of self-adjoint operator, the Fourier transform of the measure $\bar{L}_N(\cdot)$ can be written as
\begin{align*}
\hat{\tilde{L}}_N(t) &=\int_{\mathbb{R}}e^{-itx}d\bar{L}_N(x)\\
         &=\frac{1}{N}\sum_{n=1}^N\mathbb{E}\bigg( \bigg\langle e_n, e^{-it\big(A_N+D_N\big)} e_n\bigg\rangle \bigg).
\end{align*}
Since $D_N=\displaystyle\sum_{n=1}^N d_n|e_n\rangle \langle e_n|$, here $\{d_n\}_{n=1}^N$ are i.i.d Cauchy distribution with parameter $\lambda>0$ and $\{e_n\}_{n=1}^N$ is the standard basis for $\mathbb{C}^N$. Therefore, exactly the same way as it is done in (\ref{cal-ft}) (with the assumptions $\omega_2\equiv0$ and $h^0_L=\Delta_L$) we write
\begin{align*}
\hat{\tilde{L}}_N(t) &=e^{-\lambda|t|}\frac{1}{N}\sum_{n=1}^N\mathbb{E}\bigg( \bigg\langle e_n, e^{-itA_N} e_n\bigg\rangle \bigg)\\
&=e^{-\lambda|t|}\hat{L}_N(t).
\end{align*}
Now the measure $\tilde{L}_N(\cdot)$ can be written as the convolution
\begin{equation}
\label{c-eesd}
\tilde{L}_N(\cdot)=\big( \mu_1*L_N \big)(\cdot),~~d\mu_1(x)=\frac{1}{\pi}\frac{\lambda}{\lambda^2+x^2}dx,~\lambda>0.
\end{equation}
Now from the above expression it is immediate that the measure $\tilde{L}_N(\cdot)$ is absolutely continuous w.r.t the Lebesgue measure on $\mathbb{R}$ and we denote its density as $d\tilde{L}_N(x)=\tilde{\rho}_N(x)dx$. Now (\ref{conv-eesd}) will follow from the fact that $\hat{\tilde{L}}_N(t)=\hat{\tilde{\rho}}_N(t)$. \qed
~\\~\\
The uniform convergence of $\rho_L$, the local density of states functions ($\ell$DOSf) is very important to calculate the exact limit of the sequence of random measures $\{\mu^\omega_{L,E}(\cdot)\}_L$. We use the inverse Fourier transform to do that.
\begin{prop}
\label{uni-lo}
The local density of states function, $\rho_L$ converges uniformly to $\rho$, the density of sates function of $H^\omega$.
\end{prop}
\begin{proof}
Using (\ref{cal-ft}) and (\ref{ft}), we estimate the decay of $\hat{\rho}_L(t)$ and $\hat{\rho}_L(t)$ as
\begin{equation}
\label{decay-ft}
|\hat{\rho}_L(t)|\leq e^{-\lambda t} ~~\text{and}~~|\hat{\rho}(t)|\leq e^{-|\lambda| t} ~~\forall~t.
\end{equation}
The inverse Fourier transform formula give
\begin{equation}
\label{est-inv-ft}
\rho_L(x)-\rho(x)=\int_{\mathbb{R}}e^{itx}\bigg(\hat{\rho}_L(t)-\hat{\rho}(t)\bigg)dt.
\end{equation}
Since $\nu(\cdot)$ is the weak limit of $\nu_L(\cdot)$, we have the pointwise convergence of $\rho_L(t)$ to $\rho(t)$ and (\ref{decay-ft}) give the inequality $\big| \hat{\rho}_L(t)-\hat{\rho}(t) \big|\leq 2 e^{-|\lambda| t}$.\\
Now the uniform convergence of $\rho_L$ to $\rho$ follows from the dominated convergence theorem, since we write (\ref{est-inv-ft}) as
\begin{equation}
\label{est-ft}
|\rho_L(x)-\rho(x)|\leq\int_{\mathbb{R}}\big|\hat{\rho}_L(t)-\hat{\rho}(t)\big|dt.
\end{equation}
\end{proof}
\noindent Before moving to the next result, we recall the Wegner and Minami estimate (see \cite[Theorem 2.3 \& 2.1]{JFA}) for $H^\omega$. For any $\Lambda\subset \mathbb{Z}^d$, $H^\omega_\Lambda=\chi_{_\Lambda} H^\omega \chi_{_\Lambda}$ and for all bounded interval $I\subset \mathbb{R}$, we have
\begin{equation}
\label{wm}
\begin{split}
&\mathbb{E}\bigg(Tr\bigg(E_{H^\omega_{_\Lambda}}(I)\bigg)\bigg)\leq C |\Lambda||I|,~~~~C>0,\\
&\mathbb{E}\bigg(Tr\bigg(E_{H^\omega_{_\Lambda}}(I)\bigg)\bigg(Tr\bigg(E_{H^\omega_{_\Lambda}}(I)\bigg)-1\bigg)\bigg)\leq \bigg( C |\Lambda||I|\bigg)^2.
\end{split}
\end{equation}
Let $G^\Lambda(n,m;z)$ denote the Green's function of $H^\omega_\Lambda$ at the basis vectors $e_n$ and $e_m$, in other words $G^\Lambda(n,m;z)=\big\langle e_n, \big(H^\omega_\Lambda-z\big)^{-1} e_m\big\rangle,~\text{Im}~z>0$.
\\~\\
Now we divide the the box $\Lambda_L$ into $N_L^d$ numbers of disjoint cubes $C_p,~p=1,2,\cdots,N_L^d$ with side length $\frac{2L+1}{N_L},~N_L=(2L+1)^\epsilon,~0<\epsilon<1$. Define
\begin{equation*}
\begin{split}
&\partial C_p=\big\{n\in C_p:\exists ~n'\in\mathbb{Z}^d\setminus C_p ~\text{such~that}~|n-n'|=1\big\},\\
&\text{int}(C_p)=\big\{ n\in C_p:~\text{dist}(n,\partial C_p)>\ln L  \big\}.
\end{split}
\end{equation*}
Define the random measure $\eta^\omega_{E,p}(\cdot)$ associated with the eigenvalues $H^\omega_{C_p}$, the restriction of $H^\omega$ to the box $C_p$ as
\begin{equation}
\label{sm-box}
\eta^\omega_{E,p}(\cdot)=\frac{1}{(2L+1)^\beta}\sum_{E_j\in\sigma(H^\omega_{C_p})}\delta_{_{(2L+1)^\gamma(E_j-E)}}(\cdot).
\end{equation}
We also consider the superposition of $\{\eta^\omega_{p,E}(\cdot)\}_p$ and denote it by $\eta^\omega_{E, L}(\cdot)$,
\begin{equation}
\label{super-pos}
\eta^\omega_{E,L}(\cdot)=\sum_{p=1}^{N_L^d}\eta^\omega_{E,p}(\cdot).
\end{equation}
To obtain the limit of $\{\mu^\omega_{E,L}(\cdot)\}_L$ as in (\ref{meas1}), first we show that the limit of 
$\{\eta^\omega_{E,L}(\cdot)\}_L$ and $\{\mu^\omega_{E,L}(\cdot)\}_L$ are the same and then we prove that $\{\eta^\omega_{E,L}(\cdot)\}_L$ converge to a deterministic measure which is nothing but the limit of its expectation, $\big\{\mathbb{E}\big(\eta^\omega_{E,L}(\cdot)\big)\big\}_L$.
\begin{lem} 
\label{v-con}
For $0<\gamma<\frac{1}{2}$, the two sequence of measures $\{\mu^\omega_{E,L}\}$ and $\{\eta^\omega_{E,L}\}$ have the same limit points in the vague (convergence) sense, a.e $\omega$.
\end{lem}
\begin{proof}
Since the set of linear combination of the functions of the form $\phi_z(x)=\frac{1}{x-z},~\text{Im} ~z>0$, are dense in $L^1(\mathbb{R})$ and also we have $C_K(\mathbb{R}) \subset L^1(\mathbb{R})$, here $C_K(\mathbb{R})$ denote the set of continous functions on $\mathbb{R}$, with compact support, therefore to prove the lemma it is enough to verify 
\begin{equation}
\label{equi-con}
\int_\mathbb{R} \phi_z(x) ~d\mu^\omega_{E,L}(x)-\int_\mathbb{R}\phi_z(x)~d\eta^\omega_{E,L}(x)\xrightarrow{L\to\infty} 0,~\text{Im}~z>0.
\end{equation}
We refer to \cite[Appendix: The Stone-Weierstrass Gavotte]{CFKS} for more details.\\~\\
 For $n\in \text{int}(C_p ),~\text{Im}~z>0$, we have the well known resolvent identity,
 \begin{equation}
 \label{resol}
 G^{\Lambda_L}(n,n;z)=G^{C_p}(n,n;z)+\sum_{(m,k)\in\partial C_p}G^{C_p}(n,m;z)~G^\Lambda_L(k,n;z),
 \end{equation}
 where $(m,k)\in\partial C_p$ means $m\in C_p$ and $k\in\mathbb{Z}^d\setminus C_p$ such that $|m-k|=1$.
 Denote $z_L=E+(2L+1)^{-\gamma}z$, then using the above  identity (\ref{resol}) we write
\begin{equation}
\label{dif-1}
\begin{split}
&\bigg|\int_\mathbb{R} \phi_z(x) ~d\mu^\omega_{E,L}(x)-\int_\mathbb{R}\phi_z(x)~d\eta^\omega_{E,L}(x)\bigg|\\
&=\frac{1}{(2L+1)^d} \bigg | Tr \big(H^\omega_{\Lambda_L}-z_L\big)^{-1}-\sum_pTr\big(H^\omega_{C_p}-z_L\big)^{-1}\bigg|\\
&\leq \frac{1}{(2L+1)^d} \Bigg\{\sum_p\sum_{n\in C_p\setminus \text{int}(C_p)}\bigg(\big|G^{\Lambda_L}(n,n;z_L)\big|+\big|G^{C_p}(n,n;z_L)\big|\Bigg)\\
&~~~~~+ \sum_p \sum_{n\in \text{int}(C_p)}\sum_{(m,k)\in\partial C_p}\big|G^{C_p}(n,m;z_L)\big|~\big|G^{\Lambda_L}(k,n;z_L)\big|\Bigg\}.
\end{split}
\end{equation}
Now we use the Cauchy-Schwarz inequality in second term of the r.h.s of the above to get,

\begin{equation}\label{dif-2}
\begin{split}
&\bigg|\int_\mathbb{R} \phi_z(x) ~d\mu^\omega_{E,L}(x)-\int_\mathbb{R}\phi_z(x)~d\eta^\omega_{E,L}(x)\bigg|\\
&\leq \frac{1}{(2L+1)^d}\Bigg\{ \sum_p\sum_{n\in C_p\setminus \text{int} (C_p)}\bigg(\big|G^{\Lambda_L}(n,n;z_L)\big|+\big|G^{C_p}(n,n;z_L)\big|\bigg)\\
&~~~~~~~~~~~+ \sum_p \sum_{(m,k)\in\partial C_p}\Bigg(\bigg( \sum_{n\in \text{int}(C_p)}\big|G^{C_p}(n,m;z_L)\big|^2\bigg)^{\frac{1}{2}}\\
&~~~~~~~~~~~~~~~~~~~~~~~~~~~~~~~~~
~\times\bigg( \sum_{n\in \text{int}(C_p)}\big |G^{\Lambda_L}(k,n;z_L)\big|^2\bigg)^{\frac{1}{2}}\Bigg)\Bigg\}\\
&\leq \frac{1}{(2L+1)^d}\Bigg\{ \sum_p\sum_{n\in C_p\setminus \text{int} (C_p)}\bigg(\big|G^{\Lambda_L}(n,n;z_L)\big|+\big|G^{C_p}(n,n;z_L)\big|\bigg)\\
&~~~~~~~~~~~+ \sum_p \sum_{(m,k)\in\partial C_p}\bigg\|\big(H^\omega_{C_p}-z_L\big)^{-1}e_m\bigg\|~\bigg\|\big(H^\omega_{\Lambda_L}-\bar{z}_L\big)^{-1}e_k\bigg\|\Bigg\}\\
&\leq \frac{1}{(2L+1)^d}\Bigg(2N_L^d \bigg(\frac{2L+1}{N_L}\bigg)^{d-1}\ln L\big|\text{Im}~z_L\big|^{-1}\\
&~~~~~~~~~~~+N_L^d\bigg(\frac{2L+1}{N_L}\bigg)^{d-1}\ln L \big|\text{Im}~z_L\big|^{-2}\Bigg)\\
&=2(2L+1)^{-1}N_L(\ln L) (2L+1)^\gamma \text{Im} ~z\\
&~~~~~~~~~~~+(2L+1)^{-1}N_L(\ln L) (2L+1)^{2\gamma} \text{Im} ~z\\
&=(2L+1)^{-(1-\gamma-\epsilon)}(\ln L)\text{Im} ~z+L^{-(1-2\gamma-\epsilon)}(\ln L)\text{Im} ~z\\
&=(2L+1)^{-(1-2\gamma-\epsilon)}(\ln L)\text{Im} ~z\big((2L+1)^{-\gamma}+1\big).
\end{split}
\end{equation}
In the above we have used the fact that $N_L=(2L+1)^\epsilon$ and we choose $0<\epsilon<1-2\gamma$.  Since $0<\gamma<\frac{1}{2}$, the convergence in (\ref{equi-con}) is immediate.
\end{proof}
\noindent The same calculation done in Lemma \ref{v-con} will also show that the vague limit points of  $\mathbb{E}\big(\mu^\omega_{L,E}(\cdot)\big)$ and $\mathbb{E}\big(\eta_{L,E}(\cdot)\big)$ are the same.
\begin{cor}
\label{v-con-avg}
For any bounded interval $I\subset \mathbb{R}$ we have 
\begin{equation}
\label{int}
\lim_{L\to\infty}\mathbb{E}\big(\mu^\omega_{E,L}(I)\big)=\lim_{L\to\infty}\mathbb{E}\big(\eta^\omega_{E,L}(I)\big).
\end{equation}
\end{cor}
\begin{proof}
Using the estimate (\ref{dif-2}), it can be easily shown that (\ref{equi-con}) is also true for the deterministic measures 
$\mathbb{E}\big(\mu^\omega_{E,L}(\cdot)\big)$ and $\mathbb{E}\big(\eta^\omega_{E,L}(\cdot)\big)$. So both these sequence of measures have the same limit points, in the sense of vague convergence. Now the Wegner estimate (\ref{wm}) will ensure that the every limit points of $\mathbb{E}\big(\mu^\omega_{E,L}(\cdot)\big)$ is absolutely continuous, therefore we have the (\ref{int}), we refer to \cite[Lemma 4.1] {KO} for more details.
\end{proof}
\noindent  The next result shows that the limit points of the sequence of random measure 
$\{\eta^\omega_{E, L}(\cdot)\}_L$ is deterministic. Later we will compute the exact limit explicitly. 
\begin{lem}
\label{LLN}
Under the assumption of the Theorem \ref{eig-sta}, for any bounded interval $I\subset \mathbb{R}$ we have
\begin{equation}
\label{lln}
\lim_{L\to\infty}\eta^\omega_{E, L}(I)=\lim_{L\to\infty}\mathbb{E}\bigg(\eta^\omega_{E, L}(I)\bigg)~~a.e~\omega.
\end{equation}
\end{lem}
\begin{proof}
An application of Borel-Cantelli lemma will give (\ref{lln}), once  we able to show that for any $a>0$,
\begin{equation}
\label{borel-cantelli}
\sum_{L=1}^\infty\mathbb{P}\bigg(\omega:\big| \eta^\omega_{E,L}(I)- \bar{\eta}_{E,L}(I)\big| >a \bigg)<\infty,~~~\bar{\eta}_{E,L}(I)=\mathbb{E}\big(\eta^\omega_{E,L}(I)\big).
\end{equation}
We use Markov's inequality to write
\begin{equation}
\label{markov}
\sum_{L=1}^\infty\mathbb{P}\bigg(\omega:\big| \eta^\omega_{E,L}(I)- \bar{\eta}_{E,L}(I)\big| >a\bigg)\leq
\frac{1}{a^2}\sum_{L=1}^\infty\mathbb{E}\Bigg(\bigg(\eta^\omega_{E,L}(I)- \bar{\eta}_{E,L}(I)\bigg)^2\Bigg).
\end{equation}
Let $\{E_{j}^p\}_j$ be the set of all eigenvalues of $H^\omega_{C_p}$ and 
$X_{p,j}(\omega)=\chi_{_{A_{p,j}}}(\omega)$, denote the characteristic function of the set $A_{p,j}$ and the set is defined by 
$$A_{p,j}=\{\omega:E_j^p\in I_L\},~~I_L=E+(2L+1)^{-\gamma} I.$$
The measure $\eta^\omega_{E,p}(\cdot)$, defined in (\ref{sm-box}) can be written as
\begin{align}
\label{super-cha}
\eta^\omega_{E,p}(I) &=\frac{1}{(2L+1)^\beta} \sum_j X_{p,j}(\omega)~~\text{and}~~\eta^\omega_{E,L}(I)=\sum_p \eta^\omega_{E,p}(I).
\end{align}
Let $\bar{\eta}_{E,p}(I)$ denote the expectation of the random variable $\eta^\omega_{E,L}(I)$. Since the sequence random variables $\{\eta^\omega_{E,p}(I)\}_p$ are independent, we rewrite the r.h.s of (\ref{markov}) as
\begin{equation}
\label{s-ind}
\begin{split}
&\mathbb{E}\Bigg(\bigg(\eta^\omega_{E,L}(I)- \bar{\eta}_{E,L}(I)\bigg)^2\Bigg)\\
&=\sum_p\mathbb{E}\Bigg(\bigg(\eta^\omega_{E,p}(I)-\bar{\eta}_{E,p}(I)\bigg)^2\Bigg)\\
&\leq \sum_p\mathbb{E}\bigg( \eta^\omega_{E,p}(I)\bigg)^2\\
&=\frac{1}{(2L+1)^{2\beta}}\sum_p\Bigg(\sum_j\mathbb{E}\big(X_{p,j}(\omega)  \big)\\
&~~~~~~~~~~+2\sum_{n<m}\mathbb{E}  \bigg(X_{p,n}(\omega)X_{p,m}(\omega)\bigg) \Bigg)\\
&=\frac{1}{(2L+1)^{\beta}}\mathbb{E}\bigg(\eta^\omega_{E,L}(I)\bigg)\\
&~~~~~~~~~~~~+\frac{2}{(2L+1)^{2\beta}}\sum_p\sum_{n<m}\mathbb{E}  \bigg(X_{p,n}(\omega)X_{p,m}(\omega)\bigg).
\end{split}
\end{equation}
We observe the identity
\begin{equation}
\label{use-minami}
2\sum_{n<m} \bigg(X_{p,n}(\omega)X_{p,m}(\omega)\bigg)
=\sum_{j\ge 2}j(j-1)\chi_{_{\{\omega:Tr(H^\omega_{C_p}(I_L))=j\}}}(\omega).
\end{equation}
Now we use the Minami estimate (\ref{wm}) to bound the second term of the r.h.s of (\ref{s-ind}) as
\begin{equation}
\label{u-m}
\begin{split}
&\frac{2}{(2L+1)^{2\beta}}\sum_p\sum_{n<m}\mathbb{E}  \bigg(X_{p,n}(\omega)X_{p,m}(\omega)\bigg)\\
&=\frac{1}{(2L+1)^{2\beta}}\sum_p \mathbb{E}\bigg(Tr\bigg(E_{H^\omega_{_{C_P}}}(I_L)\bigg)\bigg(Tr\bigg(E_{H^\omega_{_{C_p}}}(I_L)\bigg)-1\bigg)\bigg)\\
&\leq \frac{C^2}{(2L+1)^{2\beta}}N_L^d\bigg(\frac{2L+1}{N_L}\bigg)^{2d}(2L+1)^{-2\gamma}|I|^2\\
&=C^2N_L^{-d}|I|^2,~~~~\beta+\gamma=d\\
&=C^2(2L+1)^{-d\epsilon}|I|^2,~~~~N_L=(2L+1)^\epsilon.
\end{split}
\end{equation}
The first term of the r.h.s of (\ref{s-ind}) can be estimate by the Wegner estimate (\ref{s-ind}) with help of the definitions in (\ref{sm-box}) and (\ref{super-pos})
\begin{align}
\label{u-w}
\frac{1}{(2L+1)^{\beta}}\mathbb{E}\bigg(\eta^\omega_{E,L}(I)\bigg)
&\leq\frac{C}{(2L+1)^{2\beta}}N_L^d\bigg(\frac{2L+1}{N_L}\bigg)^d(2L+1)^{-\gamma}|I|\nonumber\\
&=C(2L+1)^{-\beta}|I|,~~~~~\beta+\gamma=d.
\end{align}
We use (\ref{u-w}) and (\ref{u-m}) in (\ref{s-ind}) to get
\begin{equation}
\label{finite}
\mathbb{E}\Bigg(\bigg(\eta^\omega_{E,L}(I)- \bar{\eta}_{E,L}(I)\bigg)^2\Bigg)\leq C(2L+1)^{-\beta}|I|+
C^2(2L+1)^{-d\epsilon}|I|^2.
\end{equation}
From assumption of the Theorem \ref{eig-sta}, we have $0<\gamma<\frac{d-1}{2d}$ , therefore we can choose $\frac{1}{d}<\epsilon<1-2\gamma$ and $\beta=d-\gamma>1$ as $d\ge 2$. Now we substitute (\ref{finite}) in (\ref{markov}) to get (\ref{borel-cantelli}).
\end{proof}
\begin{rem}
We want to make a note that the proofs of Lemma \ref{v-con}, Corollary \ref{v-con-avg} and Lemma \ref{LLN} will also work for any single site distribution (SSD), provided we have the Wegner and Minami estimate as in the form (\ref{wm}).
\end{rem}
\noindent In view of the Lemma \ref{LLN}, Corollary \ref{v-con-avg}  and Lemma \ref{v-con} to prove the Theorem \ref{eig-sta}, we only need to find the limit of the sequence of deterministic measures
 $\big\{\mathbb{E}\big(\mu^\omega_{E, L}(\cdot)\big)\big\}_L$.\\~\\
{\bf Proof of the Theorem \ref{eig-sta}.} The definition of $\rho_L$, the local density of states function ($\ell$DOSf) of $H^\omega$ given in (\ref{ldosm}) together with the equivalence of convergence given in the Corollary \ref{v-con-avg} will give the limit. Now we compute, for any bounded interval $I\subset \mathbb{R}$
\begin{align}
\label{inten}
\lim_{L\to\infty}\mathbb{E}\big(\eta^\omega_{E, L}(I)\big)
&=\lim_{L\to\infty}\mathbb{E}\big(\mu^\omega_{E, L}(I)\big)\nonumber\\
&=\lim_{L\to\infty}\frac{(2L+1)^d}{(2L+1)^\beta}\nu_L(I_L),~~~I_L=E+(2L+1)^{-\gamma}I \nonumber\\
&=\lim_{L\to\infty}(2L+1)^\gamma\int_{I_L}\rho_L(x) dx,~~~\gamma+\beta=d\nonumber\\
&=\lim_{L\to\infty}\int_I\rho_L\bigg(E+(2L+1)^{-\gamma} x\bigg)dx\nonumber\\
&=\rho(E)|I|.
\end{align}
In the last line of the above we have used the uniform convergence of $\rho_L$ to $\rho$ as $L\to\infty$, as it is shown in the Proposition \ref{uni-lo}.\\
Now using the Lemma \ref{LLN} we get, for any bounded interval $I\subset\mathbb{R}$
$$\lim_{L\to\infty} \eta^\omega_{E,L}(I)=\rho(E) |I|~~a.e~\omega.$$
The above limit is enough \big(\cite[Lemma 4.1]{KO}\big) to claim that the sequence of measures $\{\eta^\omega_{E,L}(\cdot)\}_L$ converges vaguely to the measure $\rho(E)\mathcal{L}(\cdot)~a.e~\omega$, here $\mathcal{L}(\cdot)$ denote the Lebesgue measure on $\mathbb{R}$. But in the Lemma \ref{v-con}, it has already been proved that the vague limit points of the measures
$\{\eta^\omega_{E,L}(\cdot)\}_L$ and $\{\mu^\omega_{E,L}(\cdot)\}_L$ are the same a.e $\omega$, hence the theorem. 
\qed\\~\\
Now the corollary will easily follow from the properties of vague convergence of measure.\\~\\
{\bf Proof of Corollary \ref{cap-dos}:}  Since the sequence of measure $\big\{\mu^\omega_{E,L}(\cdot)\big\}_L$ converges vaguely to  $\rho(E)\mathcal{L}(\cdot)$,  $\mathcal{L}(\cdot)$ denote the Lebesgue measure on $\mathbb{R}$, then (\ref{-dos-cap}) will follow from the fact that
\begin{equation*}
\mu^\omega_{E,L}\big([a,b]\big)\xrightarrow{L\to\infty} \rho(E)(b-a)~a.e~\omega,~~\forall~a<b.
\end{equation*}
\qed
 
\appendix
\section{Appendix}
Although it is evident from the weak convergence of  
$\nu_L(\cdot)$ ($\ell$DOSm) to $\nu(\cdot)$ (DOSm) that $\hat{\rho}_L(t)$, the Fourier transform of $\ell$DOSf, converges pointwise to $\hat{\rho}(t)$, the Fourier transform DOSf . Here we can also give an explicit estimate of the difference $\big|\hat{\rho}_L(t)- \hat{\rho}(t) \big|$.
\begin{prop}
The Fourier transform of local density of states function ($\ell$DOSf) $\hat{\rho}_L(t)$ converges compact uniformly to 
$\hat{\rho}(t)$, the Fourier transform of the density of states function (DOSf) of $H^\omega$.
\end{prop}
\begin{proof}
We use the spectral theorem of self-adjoint operator to write
\begin{align}
\label{1}
\widehat{\rho}(t)-\widehat{\rho}_L(t) &=\int_{\mathbb{R}} e^{-itx}d\nu(x)-\int_{\mathbb{R}} e^{-itx}d\nu_L(x)\nonumber\\
&=\frac{1}{(2L+1)^d}\displaystyle\sum_{n\in\Lambda_L}\mathbb{E}\bigg(\bigg\langle e_n, \bigg(e^{-itH^\omega}-e^{-itH^\omega_L}\bigg)e_n\bigg\rangle\bigg).
\end{align}
Applying Duhamel's formula, we obtain
\begin{equation}
\label{2}
\bigg(e^{-itH^\omega}-e^{-itH^\omega_L}\bigg)e_n=-\int_0^t e^{-i(s-t)H^\omega} \big(H^\omega-H^\omega_L \big)e^{-isH^\omega_L}e_n~ds
\end{equation}
Now the expression in the r.h.s of (\ref{1}) can be written as
\begin{equation}
\label{3}
\begin{split}
&\sum_{n\in\Lambda_L}\bigg\langle \delta_n, \bigg(e^{-itH^\omega}-e^{-itH^\omega_L}\bigg)\delta_n\bigg\rangle \\
&=\sum_{n\in\Lambda_L}   \sum_{\substack{k\in\partial\Lambda_L \\m\notin \Lambda_L\\|m-k|=1\\}}
\int_0^t\bigg\langle \delta_k, e^{-isH^\omega_L}\delta_n\bigg\rangle \bigg\langle \delta_n, e^{-i(s-t)H^\omega}\delta_m\bigg\rangle ds.
\end{split}
\end{equation}
Using the Cauchy-Schwarz inequality (on sum over $n$) in above, we get
\begin{equation}
\label{4}
\begin{split}
&\bigg|\sum_{n\in\Lambda_L}\bigg\langle \delta_n, \bigg(e^{-itH^\omega}-e^{-itH^\omega_L}\bigg)\delta_n\bigg\rangle \bigg|\\
&\leq\sum_{\substack{k\in\partial\Lambda_L \\m\notin \Lambda_L\\|m-k|=1\\}}
\int_0^t\Bigg\{\Bigg(\sum_{n\in\Lambda_L}\bigg|\bigg\langle \delta_k, e^{-isH^\omega_L}\delta_n\bigg\rangle\bigg|^2 \Bigg)^{\frac{1}{2}}\\
&~~~~~~~~~~~~~~~\times \Bigg(\sum_{n\in\Lambda_L}\bigg|\bigg\langle \delta_k, e^{-i(s-t)H^\omega}\delta_n\bigg\rangle\bigg|^2 \Bigg)^{\frac{1}{2}}\Bigg\}ds\\
&\leq  |t| \sum_{\substack{k\in\partial\Lambda_L \\m\notin \Lambda_L\\|m-k|=1\\}}(1)\\
&=|t|~2d(2L+1)^{d-1}.
\end{split}
\end{equation}
Using (\ref{4}), (\ref{3}) and (\ref{2}) in (\ref{1}) we obtain the compact uniform convergence of $\widehat{\rho}_L(t)$ to $\widehat{\rho}(t)$, as we have
\begin{equation}
\label{5}
\bigg|\widehat{\rho}(t)-\widehat{\rho}_L(t)\bigg|\leq 2d \frac{|t|}{(2L+1)}~~\forall~t\in\mathbb{R}.
\end{equation}
\end{proof}
\begin{rem}
To prove the above Lemma, we did not assume any condition on $\mu$, the single site distribution (SSD) of $H^\omega$. Therefore estimate (\ref{5}) is valid for any probability measure $\mu$ which act as a SSD.
\end{rem}


\end{document}